\documentclass{amsart}
\usepackage[dvips,final]{graphics}
\usepackage{array}
\usepackage{arydshln}
\usepackage[makeroom]{cancel}
 \usepackage[all]{xy}
 \usepackage{url}
\usepackage{multirow, blkarray}
\usepackage{booktabs}
\usepackage{textcomp}
 \usepackage[final]{epsfig}
 \usepackage{color}
\usepackage[T1]{fontenc}      
\usepackage[english,french]{babel}
\usepackage[utf8]{inputenc}
\usepackage{blindtext}

\usepackage{amsfonts,amscd,array, mathdots, epigraph}
\usepackage{amsmath}
\usepackage{amssymb}
\usepackage{amsthm}
\usepackage{mathrsfs}
\usepackage{stmaryrd}
\usepackage{enumitem}

\usepackage{ulem}
\usepackage{tikz}
\usepackage{xcolor}
\usepackage{multicol}
\definecolor{ufogreen}{rgb}{0.24, 0.82, 0.44}

\begin{document}

\newtheorem{theorem}{Théorème}[section]
\newtheorem{theore}{Théorème}
\newtheorem{definition}[theorem]{Définition}
\newtheorem{proposition}[theorem]{Proposition}
\newtheorem{corollary}[theorem]{Corollaire}
\newtheorem*{con}{Conjecture}
\newtheorem*{remark}{Remarque}
\newtheorem*{remarks}{Remarques}
\newtheorem*{pro}{Problème}
\newtheorem*{examples}{Exemples}
\newtheorem*{example}{Exemple}
\newtheorem{lemma}[theorem]{Lemme}

\numberwithin{equation}{section}

\newcommand{\bZ}{\mathbb{Z}}
\newcommand{\bC}{\mathbb{C}}
\newcommand{\bN}{\mathbb{N}}
\newcommand{\bQ}{\mathbb{Q}}
\newcommand{\bM}{\mathbb{N^{*}}}
\newcommand{\bR}{\mathbb{R}}

\title{$\lambda$-quiddité et sous-groupes engendrés par un nombre algébrique}

\author{Flavien Mabilat}

\date{}


\address{
}
\def\emailaddrname{{\itshape Courriel}}
\email{flavien.mabilat@univ-reims.fr}

\address{Laboratoire de Mathématiques de Reims,
UMR9008 CNRS et Université de Reims Champagne-Ardenne, 
U.F.R. Sciences Exactes et Naturelles 
Moulin de la Housse - BP 1039 
51687 Reims cedex 2,
France}

\maketitle

\selectlanguage{french}
\begin{abstract}

Lors de son travail consacré aux frises de Coxeter, M. Cuntz a initié l'étude de la notion de $\lambda$-quiddité et a posé le problème de l'étude de cette dernière sur certaines sous-parties de $\bC$. Plus précisément, les $\lambda$-quiddités sont les solutions d'une équation matricielle, liée à divers objets mathématiques, que l'on cherche à résoudre sur différents ensembles. L'objectif de ce texte est d'apporter quelques nouveaux éléments de réponse au problème soulevé par M. Cuntz dans le cas de divers sous-groupes monogènes de ($\bC,+$) engendrés par un nombre algébrique. En particulier, on étudiera les cas des sous-groupes engendrés par $a+b\sqrt{k}$.
\\
\end{abstract}

\selectlanguage{english}
\begin{abstract}

During his work devoted to Coxeter's friezes, M. Cuntz initiated the study of the notion of $\lambda$-quiddity and raised the problem of the study of this over some subsets of $\bC$. More specifically, $\lambda$-quiddities are the solutions to a matrix equation, related to various mathematical objects, which we seek to solve over different sets. The aim of this text is to provide some new insights into the problem raised by M. Cuntz in the case of some cyclic subgroups of ($\bC,+$) generated by an algebraic number. In particular, we will study the cases of subgroups generated by $a+b\sqrt{k}$.
\\
\end{abstract}

\selectlanguage{french}

\thispagestyle{empty}

\textbf{Mots clés :} $\lambda$-quiddité; groupe modulaire; sous-groupe monogène; nombre algébrique
\\
\\ \indent \textbf{Classification :} 05A05
\\
\begin{flushright}
 \textit{\og L'homme intelligent se mesure à ce qu'il sait ne pas comprendre. \fg} 
\\ Édouard Herriot, \textit{Notes et Maximes}
\end{flushright}

\section{Introduction}
\label{Intro}

Apparues au tout début des années soixante-dix sous la plume du mathématicien britannique H. S. M. Coxeter afin d'étudier les formules de Gauss associées au Pentagramma mirificum (voir \cite{Cox}), les frises de Coxeter se sont rapidement affranchies de leur rôle d'outil intermédiaire. Aujourd'hui, ces dernières, qui sont des arrangements de nombres dans le plan vérifiant une relation arithmétique appelée règle unimodulaire, constituent un objet d'étude à part entière dont les ramifications couvrent une myriade de sujets différents (voir par exemple \cite{Mo1}). Ceci explique que de très nombreux travaux soient consacrés à ces objets et à leurs applications. En particulier, un des principaux levier d'action pour étudier les frises est de considérer l'équation matricielle ci-dessous :
\[M_{n}(a_{1},\ldots,a_{n}):=\begin{pmatrix}
   a_{n} & -1 \\[4pt]
    1    & 0 
   \end{pmatrix}
\begin{pmatrix}
   a_{n-1} & -1 \\[4pt]
    1    & 0 
   \end{pmatrix}
   \cdots
   \begin{pmatrix}
   a_{1} & -1 \\[4pt]
    1    & 0 
    \end{pmatrix}=-Id.\]
		
\noindent En effet, les solutions de cette équation interviennent directement dans la construction des frises de Coxeter (voir \cite{BR} et \cite{CH} proposition 2.4). Ceci a naturellement amené plusieurs auteurs, notamment M. Cuntz et V. Ovsienko, à se pencher sur la généralisation ci-dessous de l'équation qui vient d'être présentée :

\begin{equation}
\label{a}
\tag{$E$}
M_{n}(a_1,\ldots,a_n)=\pm Id.
\end{equation}

\noindent L'étude des matrices $M_{n}(a_1,\ldots,a_n)$ est d'autant plus intéressante qu'elles interviennent dans de multiples branches des mathématiques, notamment dans l'expression des réduites des fractions continues ou dans l'étude des solutions des équations de Sturm-Liouville discrètes (voir par exemple \cite{O}). Elles apparaissent également dans l'étude du groupe modulaire. En effet, toute matrice de $SL_{2}(\mathbb{Z})$ peut s'écrire sous la forme $M_{n}(a_{1},\ldots,a_{n})$ avec $a_{1},\ldots,a_{n}$ des entiers strictement positifs. Ainsi, pour chercher les différentes écritures des éléments des sous-groupes de congruence, on doit résoudre \eqref{a} sur les anneaux $\bZ/N\bZ$.
\\
\\ \indent Les solutions de cette équation matricielle sont nommées $\lambda$-quiddités. Afin d'obtenir une description complète de ces dernières sur un ensemble donné $R$, M. Cuntz a introduit une notion d'irréductibilité (voir \cite{C} et la section suivante) dont le but est de réduire l'étude des solutions de \eqref{a} à un nombre restreint d'éléments. L'objectif naturel est alors de connaître l'ensemble des $\lambda$-quiddités irréductibles sur $R$, ou de façon plus modeste d'avoir des informations sur celles-ci. 
\\
\\ \indent Dans cette optique, on dispose déjà de certains éléments. V. Ovsienko a notamment obtenu une construction récursive des solutions de \eqref{a} sur $\mathbb{N}^{*}$ (voir \cite{O} Théorèmes 1 et 2) tandis qu'un certain nombre de résultats ont été démontrés pour le cas des anneaux $\bZ/N\bZ$ (voir \cite{M1, M3, M4, M5}). Cependant, la plupart des théorèmes de classification concernent des sous-ensembles de $\bC$, en lien avec un problème ouvert soulevé par M. Cuntz (voir \cite{C} problème 4.1). On dispose ainsi d'une description complète des $\lambda$-quiddités irréductibles sur $\mathbb{N}$ (voir \cite{C} Théorème 3.1), sur $\bZ$ (voir \cite{CH} Théorème 6.2) et sur $\bZ[\alpha]$ (voir \cite{M2} Théorème 2.7). Parmi les sous-ensembles de $\bC$ que l'on peut considérer, ceux qui apparaissent à première vue comme les plus simples sont les sous-groupes monogènes. À la lueur de ce constat, plusieurs résultats concernant ces derniers ont déjà été obtenus, notamment pour ceux engendrés par $\sqrt{k}$ (voir \cite{M6}). Afin d'avoir une présentation complète, la plupart de ces résultats seront rappelés dans la section \ref{class}.
\\
\\ \indent Notre objectif dans ce texte est de poursuivre l'étude initiée dans \cite{M6} en obtenant de nouveaux résultats de classification des $\lambda$-quiddités irréductibles sur certains sous-groupes monogènes de $(\bC,+)$. Pour effectuer cela, on commencera par rappeler dans la section \ref{RP} les définitions dont on aura besoin dans la suite et on donnera quelques éléments préliminaires dans la section \ref{pre}. Ensuite, on démontrera dans la section \ref{preuve25} deux premiers résultats de classification utilisant les éléments conjugués d'un nombre algébrique, dont un concernant les $a+b\sqrt{k}$. Pour terminer, on prouvera dans la dernière partie un résultat permettant d'obtenir facilement toutes les $\lambda$-quiddités irréductibles sur un sous-groupe monogène dont le générateur appartient à une région donnée du plan complexe.

\section{Définitions et résultats principaux}
\label{RP}

Dans cette section, on va fournir les éléments essentiels de ce texte. Dans un premier temps, on va rappeler l'ensemble des définitions et des notations dont on aura besoin dans la suite. Dans un second temps, on énoncera les résultats principaux qui seront démontrés dans les sections suivantes. Dans tout ce texte, $R$ est un sous-magma de $\bC$, c'est-à-dire une sous-partie de l'ensemble des nombres complexes stable par addition. Si $w \in \bC$, on note $<w>:=\{kw, k \in \bZ\}$, le sous-groupe de $(\bC,+)$ engendré par $w$. On commence par définir formellement le concept de $\lambda$-quiddité.

\begin{definition}[\cite{C}, définition 2.2]
\label{21}

Soit $n \in \bM$. On dit que le $n$-uplet $(a_{1},\ldots,a_{n})$ d'éléments de $R$ est une $\lambda$-quiddité sur $R$ de taille $n$ si $(a_{1},\ldots,a_{n})$ est une solution de \eqref{a}, c'est-à-dire si $(a_{1},\ldots,a_{n})$ vérifie $M_{n}(a_{1},\ldots,a_{n})=\pm Id.$ En cas d'absence d'ambiguïté, on parlera simplement de $\lambda$-quiddité.

\end{definition}

\noindent En vue de l'étude des $\lambda$-quiddités, on rappelle les deux définitions ci-dessous :

\begin{definition}[\cite{C}, lemme 2.7]
\label{22}

Soient $(n,m) \in (\bM)^{2}$, $(a_{1},\ldots,a_{n})$ un $n$-uplet d'éléments de $R$ et $(b_{1},\ldots,b_{m})$ un $m$-uplet d'éléments de $R$. On définit l'opération suivante: \[(a_{1},\ldots,a_{n}) \oplus (b_{1},\ldots,b_{m}):= (a_{1}+b_{m},a_{2},\ldots,a_{n-1},a_{n}+b_{1},b_{2},\ldots,b_{m-1}).\] Le $(n+m-2)$-uplet ainsi obtenu est appelé la somme de $(a_{1},\ldots,a_{n})$ avec $(b_{1},\ldots,b_{m})$.

\end{definition}

\begin{examples} 

{\rm On considère ici $R=\bZ$. On a :} 
\begin{itemize}
\item $(-1,2,4) \oplus (3,0,1)= (0,2,7,0)$;
\item $(-2,1,3,1) \oplus (2,3,1)= (-1,1,3,3,3)$;
\item $(2,1,0,2) \oplus (1,-3,2,5,1)= (3,1,0,3,-3,2,5)$;
\item $n \geq 2$, $(a_{1},\ldots,a_{n}) \oplus (0,0) = (0,0) \oplus (a_{1},\ldots,a_{n})=(a_{1},\ldots,a_{n})$.
\end{itemize}

\end{examples}

$\oplus$ est une opération non commutative et non associative (voir \cite{WZ} exemple 2.1). En revanche, elle est particulièrement utile pour l'étude des $\lambda$-quiddités car elle vérifie la propriété  remarquable suivante : si $(b_{1},\ldots,b_{m})$ est une $\lambda$-quiddité sur $R$ alors la somme $(a_{1},\ldots,a_{n}) \oplus (b_{1},\ldots,b_{m})$ est une $\lambda$-quiddité sur $R$ si et seulement si $(a_{1},\ldots,a_{n})$ est une $\lambda$-quiddité sur $R$ (voir \cite{C,WZ} et \cite{M1} proposition 3.7).

\begin{definition}[\cite{C}, définition 2.5]
\label{23}

Soient $(a_{1},\ldots,a_{n})$ et $(b_{1},\ldots,b_{n})$ deux $n$-uplets d'éléments de $R$. On dit que $(a_{1},\ldots,a_{n}) \sim (b_{1},\ldots,b_{n})$ si $(b_{1},\ldots,b_{n})$ est obtenu par permutations circulaires de $(a_{1},\ldots,a_{n})$ ou de $(a_{n},\ldots,a_{1})$.

\end{definition}

On vérifie aisément que $\sim$ est une relation d'équivalence sur l'ensemble des $n$-uplets d'éléments de $R$ (voir \cite{WZ}, lemme 1.7). Par ailleurs, si un $n$-uplet d'éléments de $R$ est une solution de \eqref{a} alors tout $n$-uplet d'éléments de $R$ qui lui est équivalent est aussi une solution de \eqref{a} (voir \cite{C} proposition 2.6). Muni de ces définitions, on peut définir la notion d'irréductibilité  précédemment évoquée.

\begin{definition}[\cite{C}, définition 2.9]
\label{24}

Une $\lambda$-quiddité $(c_{1},\ldots,c_{n})$ sur $R$ avec $n \geq 3$ est dite réductible s'il existe une $\lambda$-quiddité $(b_{1},\ldots,b_{l})$ sur $R$ et un $m$-uplet $(a_{1},\ldots,a_{m})$ d'éléments de $R$ tels que \begin{itemize}
\item $(c_{1},\ldots,c_{n}) \sim (a_{1},\ldots,a_{m}) \oplus (b_{1},\ldots,b_{l})$,
\item $m \geq 3$ et $l \geq 3$.
\end{itemize}
Une $\lambda$-quiddité est dite irréductible si elle n'est pas réductible.
\end{definition}

\begin{remark} 

{\rm Dans le cas où $0 \in R$, $(0,0)$ est une $\lambda$-quiddité sur $R$. Cependant, celle-ci n'est jamais considérée comme étant irréductible.}

\end{remark}

Une fois cette notion définie, l'objectif principal de l'étude de l'équation \eqref{a} est d'obtenir une description complète des solutions irréductibles. On possède déjà un certain nombre de résultats de classification des $\lambda$-quiddités irréductibles, dont plusieurs abordent le cas des sous-groupes monogènes (voir \cite{C,M2,M6} ainsi que la section \ref{class}). En particulier, l'un de ces résultats résout entièrement la question de la classification pour le cas des sous-groupes engendrés par un nombre transcendant, c'est-à-dire un nombre qui n'est racine d'aucun polynôme non nul à coefficients entiers. Notre premier objectif est donc naturellement d'avoir des informations pour les sous-groupes engendrés par un nombre algébrique. Cette étude nous permettra notamment de démontrer le résultat suivant.

\begin{theorem}
\label{25}

Soit $w \in \bC$ algébrique. Si $w$ a un élément conjugué de module supérieur à 2 alors l'ensemble des $\lambda$-quiddités irréductibles sur $<w>$ est : \[\{(0,kw,0,-kw), (kw,0,-kw,0); k \in \bZ\}.\]

\end{theorem}

\noindent Avec ce résultat, on déduira le résultat de classification ci-dessous :

\begin{theorem}
\label{26}

Soit $(a,b,k) \in \bZ^{3}$ avec $a \neq 0$, $b \neq 0$ et $k \geq 2$. On suppose que $a+b\sqrt{k} \neq \pm 1$. L'ensemble des $\lambda$-quiddités irréductibles sur $<a+b\sqrt{k}>$ est :
\[\{(0,l(a+b\sqrt{k}),0,-l(a+b\sqrt{k})), (l(a+b\sqrt{k}),0,-l(a+b\sqrt{k}),0), l \in \bZ\}.\]

\end{theorem}

Ces deux théorèmes sont démontrés dans la section \ref{preuve25}. Notons par ailleurs que les résultats de classification déjà connus (voir section \ref{class}) permettent de considérer les cas exclus dans le théorème précédent et donc de régler entièrement la question de la classification pour ces sous-groupes.
\\
\\ \indent On s'intéressera pour terminer aux cas de certains nombres complexes non réels en démontrant le résultat ci-dessous :

\begin{theorem}
\label{27}

Soit $w=a+ib \in \bC$. On suppose que $\left|ab\right| \geq 1$. L'ensemble des $\lambda$-quiddités irréductibles sur $<w>$ est :
\[\{(0,kw,0,-kw), (kw,0,-kw,0); k \in \bZ\}.\]

\end{theorem}

\noindent Ce théorème est prouvé dans la section \ref{preuve27}.

\section{Résultats préliminaires}
\label{pre}

Le but de cette partie est de collecter un certain nombre d'éléments utiles pour la suite et d'énoncer plusieurs théorèmes de classification précédemment évoqués. On en profitera également pour fournir quelques éléments plus généraux sur les solutions de l'équation \eqref{a}.

\subsection{Premiers résultats}
\label{Rp}

On débute cette section en donnant la liste des $\lambda$-quiddités sur $\bC$ lorsque $1 \leq n \leq 4$.

\begin{proposition}[\cite{CH}, exemple 2.7]
\label{31}
\begin{itemize}
\item \eqref{a} n'a pas de solution de taille 1.
\item $(0,0)$ est la seule solution de \eqref{a} de taille 2.
\item $(1,1,1)$ et $(-1,-1,-1)$ sont les seules solutions de \eqref{a} de taille 3.
\item Les solutions de \eqref{a} pour $n=4$ sont les 4-uplets suivants $(-a,b,a,-b)$ avec $ab=0$ (c'est-à-dire $a=0$ ou $b=0$) et $(a,b,a,b)$ avec $ab=2$.
\end{itemize}
 
\end{proposition}

Pour obtenir les $\lambda$-quiddités de petite taille sur un sous-magma $R$ de $\bC$, il suffit d'adapter la proposition précédente aux éléments appartenant à $R$. Par exemple, si $0 \notin R$ alors il n'existe pas de $\lambda$-quiddité de taille 2 sur $R$. 

\begin{proposition}[\cite{M6}, proposition 3.3]
\label{32}

Soient $G$ un sous-groupe de $\bC$. 
\\i) Une $\lambda$-quiddité sur $G$ de taille supérieure à 5 contenant 0 est réductible.
\\ii) Si $1 \notin G$ alors les $\lambda$-quiddités sur $G$ de taille 4 sont irréductibles.
 
\end{proposition}

Afin de classifier les solutions de \eqref{a}, on a besoin d'informations sur les composantes de ces dernières. Pour cela, on dispose notamment du théorème ci-dessous qui est un ingrédient central des preuves de la plupart des résultats de classification sur les sous-ensembles de $\bC$ :

\begin{theorem}[Cuntz-Holm, \cite{CH} corollaire 3.3]
\label{33}

Soit $(a_{1},\ldots,a_{n}) \in \bC^{n}$ une $\lambda$-quiddité. Il existe $(i,j)$ dans $[\![1;n]\!]^{2}$, $i \neq j$, tels que $\left|a_{i}\right| < 2$ et $\left|a_{j}\right| < 2$.

\end{theorem}

\begin{remarks}

{\rm i) La constante 2 donnée dans le théorème précédent est optimale. En effet, pour tout $\epsilon \in ]0,2]$, on peut construire une $\lambda$-quiddité sur $\bR$ dont toutes les composantes sont de module supérieur à $2-\epsilon$ (voir \cite{M2} proposition 3.5).
\\
\\ii) En revanche, il n'existe pas de constante $s$ pour laquelle il existe nécessairement une composante plus grande que $s$, même si ces dernières sont toutes non nulles. En effet, posons pour tout $n \geq 2$, $v_{n}:=2{\rm sin}(\frac{\pi}{n})$ et $B_{n}:=\begin{pmatrix}
   v_{n} & -1 \\[4pt]
    1    & 0 
   \end{pmatrix}$.
\\
\\Soit $\epsilon > 0$. On a $\lim\limits_{n \rightarrow +\infty} v_{n}=0$. Ainsi, il existe $n \in \bN$, $n \geq 3$, tel que $0~<~v_{n}~<~\epsilon$. Le polynôme caractéristique de $B_{n}$ est $\chi_{B_{n}}(X)={\rm det}(B_{n}-XId)=X^{2}-v_{n}X+1$. Le discriminant de ce polynôme est $\Delta=v_{n}^{2}-4~<~0$. Donc, $\chi_{B_{n}}$ a deux racines complexes conjuguées $x_{1}$ et $x_{2}$ avec :

\begin{eqnarray*}
x_{1} &=& \frac{v_{n}+i\sqrt{4-v_{n}^{2}}}{2} \\
      &=& {\rm sin}\left(\frac{\pi}{n}\right)+i\sqrt{\left({\rm cos}\left(\frac{\pi}{n}\right)\right)^{2}} \\
			&=& {\rm sin}\left(\frac{\pi}{n}\right)+i~{\rm cos}\left(\frac{\pi}{n}\right)~{\rm (car }~{\rm cos}\left(\frac{\pi}{n}\right) \geq 0) \\
			&=& i e^{-\frac{i\pi}{n}}. \\
\end{eqnarray*}

\noindent $B_{n}$ a deux valeurs propres distinctes : $i e^{-\frac{i\pi}{n}}$ et $-i e^{\frac{i\pi}{n}}$. Ainsi, $B_{n}$ est diagonalisable et $(B_{n})^{2n}=(-1)^{n}Id$. Donc, le $2n$-uplet formé uniquement de $2{\rm sin}(\frac{\pi}{n})$ est une $\lambda$-quiddité sur $\bR$ dont toutes les composantes sont non nulles et inférieures à $\epsilon$.
}

\end{remarks}

Dans la suite, on utilisera également une formule donnant les coefficients de la matrice $M_{n}(a_{1},\ldots,a_{n})$ à l'aide de déterminants. Pour cela, on donne les notations suivantes :
\\
\\On pose $K_{-1}:=0$ et $K_{0}:=1$. Soient $n \in \bM$ et $(a_{1},\ldots,a_{n}) \in \bC^{n}$. On note \[K_{n}(a_{1},\ldots,a_{n}):=
\left|
\begin{array}{cccccc}
a_1&1&&&\\[4pt]
1&a_{2}&1&&\\[4pt]
&\ddots&\ddots&\!\!\ddots&\\[4pt]
&&1&a_{n-1}&\!\!\!\!\!1\\[4pt]
&&&\!\!\!\!\!1&\!\!\!\!a_{n}
\end{array}
\right|.\] $K_{n}(a_{1},\ldots,a_{n})$ est le continuant de $a_{1},\ldots,a_{n}$. On dispose de l'égalité suivante (voir \cite{CO}) : Soient $n \in \bM$ et $(a_{1},\ldots,a_{n}) \in \bC^{n}$.
\[M_{n}(a_{1},\ldots,a_{n})=\begin{pmatrix}
    K_{n}(a_{1},\ldots,a_{n}) & -K_{n-1}(a_{2},\ldots,a_{n}) \\
    K_{n-1}(a_{1},\ldots,a_{n-1})  & -K_{n-2}(a_{2},\ldots,a_{n-1}) 
   \end{pmatrix}.\]

Pour obtenir une expression du continuant sous la forme d'un polynôme, on peut utiliser un algorithme simple nommé algorithme d'Euler. Celui-ci fonctionne de la manière suivante (voir par exemple \cite{CO} section 2). $K_{n}(a_{1},\ldots,a_{n})$ est la somme de tous les produits possibles de $a_{1},\ldots,a_{n}$ dans lesquels un nombre quelconque de paires disjointes de termes consécutifs est supprimé. Chacun de ces produits étant multiplié par (-1) puissance le nombre de paires supprimées. On considère donc d'abord le produit $a_{1} \times \ldots \times a_{n}$. Puis, on soustrait tous les produits de la forme $a_{1} \times \ldots a_{i-1} \times a_{i+2} \times \ldots \times a_{n}$. Après, on ajoute tous les produits possibles de $a_{1},\ldots,a_{n}$, dans lesquels deux paires disjointes de termes consécutifs ont été supprimées et on poursuit de la même façon. 
\\
\\Cela donne, par exemple, pour $n=4$ : $K_{4}(a_{1},a_{2},a_{3},a_{4})=a_{1}a_{2}a_{3}a_{4}-a_{3}a_{4}-a_{1}a_{4}-a_{1}a_{2}+1$.

\subsection{Résultats de classification}
\label{class}

Comme on l'a évoqué dans les sections précédentes, on dispose déjà d'un certain nombre de théorèmes de classification sur des sous-ensembles $R$ de $\bC$. Dans les cas où $R$ est un sous-anneau, on a les deux résultats ci-dessous :

\begin{theorem}[Cuntz-Holm, \cite{CH} Théorème 6.2]
\label{34}

L'ensemble des $\lambda$-quiddités irréductibles sur $\mathbb{Z}$ est : \[\{(1,1,1), (-1,-1,-1), (0,m,0,-m), (m,0,-m,0); m \in \mathbb{Z}-\{\pm 1\} \}.\]

\end{theorem}

\begin{theorem}[\cite{M2}, Théorème 2.7]
\label{35}

Soit $\alpha$ un nombre complexe transcendant. L'ensemble des $\lambda$-quiddités irréductibles sur l'anneau $\bZ[\alpha]$ est : \[\{(1,1,1), (-1,-1,-1), (0,P(\alpha),0,-P(\alpha)), (P(\alpha),0,-P(\alpha),0); P \in \bZ[X]-\{\pm 1\} \}.\]

\end{theorem}

En plus du résultat sur $\mathbb{Z}=<1>$, on possède plusieurs résultats concernant les cas des sous-groupes monogènes. Le premier d'entre-eux est une conséquence simple du théorème \ref{33}.

\begin{proposition}[\cite{M6}, proposition 3.10]
\label{36}

Soit $w$ un nombre complexe vérifiant $\left|w\right| \geq 2$. L'ensemble des $\lambda$-quiddités irréductibles sur $<w>$ est : \[\{(0,kw,0,-kw), (kw,0,-kw,0); k \in \bZ\}.\]

\end{proposition}

Notons que cette proposition combiné au théorème \ref{34} permet d'effectuer la classification des $\lambda$-quiddités irréductibles sur les sous-groupes engendrés par un entier (voir \cite{M6} corollaire 3.11).
\\
\\ \indent En ce qui concerne les cas des sous-groupes monogènes engendrés par un nombre complexe transcendant, on dispose d'une classification complète.

\begin{proposition}[\cite{M6}, proposition 3.12]
\label{37}

Soit $\alpha$ un nombre complexe transcendant. L'ensemble des $\lambda$-quiddités irréductibles sur $<\alpha>$ est : \[\{(0,k\alpha,0,-k\alpha), (k\alpha,0,-k\alpha,0); k \in \bZ\}.\]

\end{proposition}

\noindent Dans le cas des nombres algébriques, on a le résultat suivant :

\begin{theorem}[\cite{M6}, Théorème 2.5]
\label{38}

Soit $k \in \bN$.
\\
\\i) Si $k=0$. $(0,0,0,0)$ est la seule $\lambda$-quiddité irréductible sur $<\sqrt{k}>$.
\\
\noindent ii) Si $k=2$. L'ensemble des $\lambda$-quiddités irréductibles sur $<\sqrt{2}>$ est : \[\{(\sqrt{2},\sqrt{2},\sqrt{2},\sqrt{2}), (-\sqrt{2},-\sqrt{2},-\sqrt{2},-\sqrt{2}), (0,a\sqrt{2},0,-a\sqrt{2}), (a\sqrt{2},0,-a\sqrt{2},0); a \in \bZ\}.\]

\noindent iii) Si $k=3$. L'ensemble des $\lambda$-quiddités irréductibles sur $<\sqrt{3}>$ est : \[\{\pm (\sqrt{3},\sqrt{3},\sqrt{3},\sqrt{3},\sqrt{3},\sqrt{3}), (0,a\sqrt{3},0,-a\sqrt{3}), (a\sqrt{3},0,-a\sqrt{3},0); a \in \bZ\}.\]

\noindent iv) Si $k \geq 4$. L'ensemble des $\lambda$-quiddités irréductibles sur $<\sqrt{k}>$ est : \[\{(0,a\sqrt{k},0,-a\sqrt{k}), (a\sqrt{k},0,-a\sqrt{k},0); a \in \bZ\}.\]

\end{theorem}

\noindent Pour d'autres résultats de classification, on peut consulter \cite{C,M6}.

\subsection{Solutions de taille paire}

L'objectif de cette section est de fournir un certain nombre d'éléments permettant de montrer que toutes les $\lambda$-quiddités sur un sous-groupe $G$ sont de taille paire. On commence par le résultat suivant :

\begin{proposition}
\label{39}

Soit $G$ un sous-groupe de $\bC$. Si toutes les $\lambda$-quiddités de taille paire sur $G$ ont une de leurs composantes égale à 0 alors toutes les $\lambda$-quiddités sur $G$ sont de taille paire et l'ensemble des $\lambda$-quiddités irréductibles sur $G$ est : \[\{(0,g,0,-g), (g,0,-g,0); g \in G\}.\]

\end{proposition}

\begin{proof}

Soient $n \in \mathbb{N}$, $n$ impair et $(a_{1},\ldots,a_{n}) \in G^{n}$ une $\lambda$-quiddité. 
\[M_{2n}(a_{1},\ldots,a_{n},a_{1},\ldots,a_{n})=M_{n}(a_{1},\ldots,a_{n})^{2}=Id.\]

\noindent Ainsi, $(a_{1},\ldots,a_{n},a_{1},\ldots,a_{n})$ est une solution de \eqref{a}. Par hypothèse, elle contient 0 et donc $(a_{1},\ldots,a_{n})$ contient 0. On en déduit que toutes les $\lambda$-quiddités sur $G$ contiennent 0.
\\
\\ Par la proposition \ref{32}, toutes les solutions de \eqref{a} sur $G$ de taille supérieure à 5 sont réductibles. De plus, $1 \notin G$, sinon $(1,1,1)$ serait une solution sur $G$ ne contenant pas 0. Comme $G$ est un sous-groupe de $\bC$, $-1 \notin G$. Donc, les $\lambda$-quiddités sur $G$ de taille 4 sont irréductibles (proposition \ref{32}) et il n'en existe pas de taille 3 (proposition \ref{31}). Enfin, puisque toutes les solutions contiennent 0, il n'existe pas de solution de la forme $(a,b,a,b)$ avec $ab=2$.
\\
\\Ainsi, les $\lambda$-quiddités irréductibles sur $G$ sont celles données dans l'énoncé, ce qui implique nécessairement que toutes les solutions sur $G$ sont de taille paire.

\end{proof}

\begin{remarks}

{\rm
i) La réciproque est fausse. Par exemple, toutes les solutions de \eqref{a} sur $<\sqrt{2}>$ sont de taille paire mais elles ne contiennent pas toutes 0 (voir Théorème \ref{38}).
\\
\\ii) Il existe bien sûr des cas où des $\lambda$-quiddités de taille impaire existent, comme pour $G=\bZ$. On peut également montrer que de telles solutions existent pour $G=<\varphi>$ avec $\varphi:=\frac{1+\sqrt{5}}{2}$ le nombre d'or. En effet, on montre, par un simple calcul, que $(\varphi,\varphi,\varphi,\varphi,\varphi)$ et $(-\varphi,-\varphi,-\varphi,-\varphi,-\varphi)$ sont des $\lambda$-quiddités sur $<\varphi>$. Malheureusement, on ne possède pas de classification des solutions de \eqref{a} irréductibles sur ce sous-groupe. Toutefois, on peut émettre la conjecture suivante :
}

\begin{con}

L'ensemble des $\lambda$-quiddités irréductibles sur $<\varphi>$ est : \[\{(\varphi,\varphi,\varphi,\varphi,\varphi), (-\varphi,-\varphi,-\varphi,-\varphi,-\varphi), (0,k\varphi,0,-k\varphi), (k\varphi,0,-k\varphi,0); k \in \bZ\}.\]

\end{con}

\end{remarks}

On peut également montrer que les solutions de \eqref{a} sur un ensemble donné sont toutes de taille paire en s'intéressant directement à la forme des matrices $M_{n}(a_{1},\ldots,a_{n})$. On propose ici deux exemples. 

\begin{proposition}
\label{310}

i) Les $\lambda$-quiddités sur $<e^{\frac{i \pi}{4}}>$ sont de taille paire. 
\\ii) Les $\lambda$-quiddités sur $<\frac{1}{\sqrt{2}}>$ sont de taille paire. 

\end{proposition}

\begin{proof}

i) Soit $n$  un entier naturel impair. Supposons par l'absurde qu'il existe une $\lambda$-quiddité $(a_{1},\ldots,a_{n})$ sur $<e^{\frac{i \pi}{4}}>$, avec pour tout $i \in [\![1;n]\!]$ $a_{i}=k_{i}e^{\frac{i \pi}{4}}$ ($k_{i} \in \bZ$). Il existe $\epsilon \in \{-1,1\}$ tel que 
\[\epsilon Id = M_{n}(a_{1},\ldots,a_{n})=\begin{pmatrix}
    K_{n}(a_{1},\ldots,a_{n}) & -K_{n-1}(a_{2},\ldots,a_{n}) \\
    K_{n-1}(a_{1},\ldots,a_{n-1})  & -K_{n-2}(a_{2},\ldots,a_{n-1}) 
   \end{pmatrix}.\]
	
\noindent En particulier, on a $K_{n}(a_{1},\ldots,a_{n})=\epsilon$. Or, par l'algorithme d'Euler, $K_{n}(a_{1},\ldots,a_{n})$ est une somme dont chacun des termes est un produit d'un nombre impair de $a_{i}$ (et éventuellement de -1). On factorise cette somme par $e^{\frac{i \pi}{4}}$ et on regroupe dans chaque terme les $e^{\frac{i \pi}{4}}$ restants (qui sont en nombre pair). Chaque terme se réduit donc à un entier multiplié par $\pm 1$ ou par $\pm i$. Ainsi, $K_{n}(a_{1},\ldots,a_{n})$ est de la forme $e^{\frac{i \pi}{4}}(u+iv)$ avec $(u,v) \in \bZ^{2}$. 
\\
\\De plus, $e^{\frac{i \pi}{4}}(u+iv)=K_{n}(a_{1},\ldots,a_{n})=\epsilon$, donc $u \neq 0$ et $v \neq 0$. Ainsi, \[1=\left|\epsilon\right|=\left|e^{\frac{i \pi}{4}}(u+iv)\right|=\left|u+iv\right|=\sqrt{u^{2}+v^{2}} \geq \sqrt{2}.\]

\noindent Ceci est absurde. Donc, les $\lambda$-quiddités sur $<e^{\frac{i \pi}{4}}>$ sont de taille paire.
\\
\\ii) Soit $n$  un entier naturel impair. Supposons par l'absurde qu'il existe une $\lambda$-quiddité $(a_{1},\ldots,a_{n})$ sur $<\frac{1}{\sqrt{2}}>$, avec pour tout $i \in [\![1;n]\!]$ $a_{i}=\frac{k_{i}}{\sqrt{2}}$ ($k_{i} \in \bZ$). Il existe $\epsilon \in \{-1,1\}$ tel que 
\[\epsilon Id = M_{n}(a_{1},\ldots,a_{n})=\begin{pmatrix}
    K_{n}(a_{1},\ldots,a_{n}) & -K_{n-1}(a_{2},\ldots,a_{n}) \\
    K_{n-1}(a_{1},\ldots,a_{n-1})  & -K_{n-2}(a_{2},\ldots,a_{n-1}) 
   \end{pmatrix}.\]
	
\noindent	En particulier, on a $K_{n}(a_{1},\ldots,a_{n})=\epsilon$. Or, par l'algorithme d'Euler, $K_{n}(a_{1},\ldots,a_{n})$ est une somme dont chacun des termes est un produit d'un nombre impair de $a_{i}$ (et éventuellement de -1). On factorise cette somme par $\frac{1}{\sqrt{2}}$ et on regroupe dans chaque terme les $\frac{1}{\sqrt{2}}$ restants (qui sont en nombre pair). Chaque terme se réduit donc à un entier multiplié par une puissance de $\frac{1}{2}$. Ainsi, il existe un rationnel $x$ tel que $K_{n}(a_{1},\ldots,a_{n})=\frac{1}{\sqrt{2}}x$.
\\
\\Donc, $K_{n}(a_{1},\ldots,a_{n}) \in \bR-\bQ$. Or, $K_{n}(a_{1},\ldots,a_{n})=\epsilon \in \bQ$. Ceci est absurde. Ainsi, les $\lambda$-quiddités sur $<\frac{1}{\sqrt{2}}>$ sont de taille paire.

\end{proof}

Nous allons maintenant tenter d'obtenir dans les sections suivantes de nouveaux résultats de classification des $\lambda$-quiddités irréductibles sur des sous-groupes monogènes engendrés par des nombres algébriques.

\section{Utilisation des éléments conjugués d'un nombre algébrique}
\label{preuve25}

L'objectif de cette section est d'utiliser le caractère algébrique de certains nombres pour démontrer des résultats de classification des $\lambda$-quiddités irréductibles, en particulier ceux donnés dans les théorèmes \ref{25} et \ref{26}.

\subsection{Démonstration des théorèmes \ref{25} et \ref{26}}

On commence cette sous-partie en rappelant quelques éléments qui nous seront utiles pour la suite (voir \cite{G} chapitre III, section 1.3).
\\
\\Si $\alpha$ est un nombre complexe algébrique alors :
\begin{itemize}
\item Il existe un unique polynôme unitaire $P \in \bQ[X]$ tel que si $R \in \bQ[X]$ vérifie $R(\alpha)=0$ alors $R$ est un multiple de $P$ dans $\bQ[X]$.
\item Ce polynôme $P$ est irréductible sur $\bQ[X]$ et on l'appelle le polynôme minimal de $\alpha$ sur $\bQ$.
\item Les racines sur $\bC$ du polynôme minimal de $\alpha$ sont les éléments conjugués de $\alpha$.
\end{itemize}

\noindent Notons par ailleurs que l'ensemble des nombres complexes algébriques est infini et dénombrable (voir par exemple \cite{Ca} et \cite {G} corollaire III.57).
\\
\\On peut maintenant démontrer le résultat ci-dessous :

\begin{proposition}
\label{41}

Soient $\alpha$ un nombre complexe algébrique et $\beta$ un de ses éléments conjugués. On note $\Omega(\alpha)$ (resp. $\Omega(\beta)$) l'ensemble des $\lambda$-quiddités sur $<\alpha>$ (resp. sur $<\beta>$). Si $\Omega(\alpha) \neq \emptyset$ alors $\Omega(\beta) \neq \emptyset$ et l'application :
\[\begin{array}{ccccc} 
\theta & : & \Omega(\alpha) & \longrightarrow & \Omega(\beta) \\
 & & (k_{1}\alpha,\ldots,k_{n}\alpha) & \longmapsto & (k_{1}\beta,\ldots,k_{n}\beta)  \\
\end{array}\] est une bijection.
\\
\\De plus, $\theta$ établit également une bijection entre l'ensemble des $\lambda$-quiddités irréductibles sur $<\alpha>$ et l'ensemble des $\lambda$-quiddités irréductibles sur $<\beta>$.

\end{proposition}

\begin{proof}

Si $\alpha \in \bQ$ alors $\alpha=\beta$ et le résultat est vrai. On suppose donc que $\alpha \notin \bQ$, en particulier $\alpha, \beta \neq 0$.
\\
\\Soient $P$ le polynôme minimal de $\alpha$ sur $\bQ$ et $(k_{1}\alpha,\ldots,k_{n}\alpha) \in \Omega(\alpha)$. Il existe $\epsilon \in \{\pm 1\}$ tel que 
\[\epsilon Id = M_{n}(k_{1}\alpha,\ldots,k_{n}\alpha) = \begin{pmatrix}
    K_{n}(k_{1}\alpha,\ldots,k_{n}\alpha) & -K_{n-1}(k_{2}\alpha,\ldots,k_{n}\alpha) \\
    K_{n-1}(k_{1}\alpha,\ldots,k_{n-1}\alpha)  & -K_{n-2}(k_{2}\alpha,\ldots,k_{n-1}\alpha) 
   \end{pmatrix}.\]

\noindent Posons $R(X):=K_{n}(k_{1}X,\ldots,k_{n}X)-\epsilon$. Par l'algorithme d'Euler, $R(X)$ est un polynôme en $X$. De plus, $R(\alpha)=0$. Donc, $R$ est un multiple de $P$ dans $\bQ[X]$, c'est-à-dire qu'il existe $U \in \bQ[X]$ tel que $R=UP$. En particulier, $R(\beta)=U(\beta)P(\beta)=0$. Ainsi, \[K_{n}(k_{1}\beta,\ldots,k_{n}\beta)=\epsilon.\]

\noindent On montre de la même façon que :
\begin{itemize}
\item $K_{n-1}(k_{1}\beta,\ldots,k_{n-1}\beta)=K_{n-1}(k_{2}\beta,\ldots,k_{n}\beta)=0$;
\item $K_{n-2}(k_{2}\beta,\ldots,k_{n-1}\beta)=-\epsilon$.
\end{itemize}

\noindent Donc, $(k_{1}\beta,\ldots,k_{n}\beta)$ est une solution de \eqref{a} sur $<\beta>$. Ainsi, $\Omega(\beta) \neq \emptyset$ et $\theta$ est bien définie.
\\
\\Si $(k_{1}\beta,\ldots,k_{n}\beta) \in \Omega(\beta)$ alors, comme $\alpha$ est un élément conjugué de $\beta$, on a, par ce qui précède, $(k_{1}\alpha,\ldots,k_{n}\alpha) \in \Omega(\alpha)$ et $\theta(k_{1}\alpha,\ldots,k_{n}\alpha)=(k_{1}\beta,\ldots,k_{n}\beta)$. Ainsi, $\theta$ est surjective.
\\
\\Soient $((k_{1}\alpha,\ldots,k_{n}\alpha),(k_{1}'\alpha,\ldots,k_{n}'\alpha)) \in \Omega(\alpha)^{2}$ tels que $\theta(k_{1}\alpha,\ldots,k_{n}\alpha)=\theta(k_{1}'\alpha,\ldots,k_{n}'\alpha)$. Pour tout $i \in [\![1;n]\!]$, $k_{i}\beta=k_{i}'\beta$. Comme $\beta \neq 0$, $k_{i}=k_{i}'$ et $(k_{1}\alpha,\ldots,k_{n}\alpha)=(k_{1}'\alpha,\ldots,k_{n}'\alpha)$, c'est-à-dire $\theta$ est injective. Ainsi, $\theta$ est bijective.
\\
\\On va maintenant montrer que $\theta$ établit une bijection entre l'ensemble des $\lambda$-quiddités irréductibles sur $<\alpha>$ et l'ensemble des $\lambda$-quiddités irréductibles sur $<\beta>$.
\\
\\Soit $(k_{1}\alpha,\ldots,k_{n}\alpha) \in \Omega(\alpha)$ tel que $\theta(k_{1}\alpha,\ldots,k_{n}\alpha)$ est réductible. Il existe $l, l' \geq 3$, $(a_{1}\beta,\ldots,a_{l}\beta) \in \Omega(\beta)$ et $(b_{1}\beta,\ldots,b_{l'}\beta) \in \Omega(\beta)$ tels que :

\[(k_{1}\beta,\ldots,k_{n}\beta) \sim (a_{1}\beta,\ldots,a_{l}\beta) \oplus (b_{1}\beta,\ldots,b_{l'}\beta)= ((a_{1}+b_{l'})\beta,a_{2}\beta,\ldots,a_{l-1}\beta,(a_{l}+b_{1})\beta,b_{2}\beta,\ldots,b_{l'-1}\beta).\]

\noindent On a :

\[(k_{1}\alpha,\ldots,k_{n}\alpha) \sim ((a_{1}+b_{l'})\alpha,a_{2}\alpha,\ldots,a_{l-1}\alpha,(a_{l}+b_{1})\alpha,b_{2}\alpha,\ldots,b_{l'-1}\alpha)=(a_{1}\alpha,\ldots,a_{l}\alpha) \oplus (b_{1}\alpha,\ldots,b_{l'}\alpha).\]

\noindent Comme $(b_{1}\beta,\ldots,b_{l'}\beta) \in \Omega(\beta)$, $(b_{1}\alpha,\ldots,b_{l'}\alpha) \in \Omega(\alpha)$ et donc $(k_{1}\alpha,\ldots,k_{n}\alpha)$ est réductible. Ainsi, l'image d'une solution irréductible par $\theta$ est irréductible.
\\
\\Soit $(k_{1}\alpha,\ldots,k_{n}\alpha) \in \Omega(\alpha)$ réductible. Il existe $l, l' \geq 3$, $(a_{1}\alpha,\ldots,a_{l}\alpha) \in \Omega(\alpha)$ et $(b_{1}\alpha,\ldots,b_{l'}\alpha) \in \Omega(\alpha)$ tels que :

\[(k_{1}\alpha,\ldots,k_{n}\alpha) \sim (a_{1}\alpha,\ldots,a_{l}\alpha) \oplus (b_{1}\alpha,\ldots,b_{l'}\alpha)=((a_{1}+b_{l'})\alpha,a_{2}\alpha,\ldots,a_{l-1}\alpha,(a_{l}+b_{1})\alpha,b_{2}\alpha,\ldots,b_{l'-1}\alpha).\]

\noindent On a :

\[\theta(k_{1}\alpha,\ldots,k_{n}\alpha) = (k_{1}\beta,\ldots,k_{n}\beta) \sim ((a_{1}+b_{l'})\beta,a_{2}\beta,\ldots,a_{l-1}\beta,(a_{l}+b_{1})\beta,b_{2}\beta,\ldots,b_{l'-1}\beta).\]
\noindent Donc, $\theta(k_{1}\alpha,\ldots,k_{n}\alpha)=(a_{1}\beta,\ldots,a_{l}\beta) \oplus (b_{1}\beta,\ldots,b_{l'}\beta)$.
\\
\\Comme $(b_{1}\alpha,\ldots,b_{l'}\alpha) \in \Omega(\alpha)$, $(b_{1}\beta,\ldots,b_{l'}\beta) \in \Omega(\beta)$ et donc $\theta(k_{1}\alpha,\ldots,k_{n}\alpha)$ est réductible. Ainsi, $\theta$ établit une surjection entre les éléments irréductibles de $\Omega(\alpha)$ et ceux de $\Omega(\beta)$.
\\
\\Donc, $\theta$ établit une bijection entre l'ensemble des $\lambda$-quiddités irréductibles sur $<\alpha>$ et l'ensemble des $\lambda$-quiddités irréductibles sur $<\beta>$.

\end{proof}

\noindent Muni de cette proposition, on peut maintenant prouver les théorèmes \ref{25} et \ref{26}.

\begin{proof}[Démonstration du théorème \ref{25}]

Soit $\alpha$ un nombre algébrique. On suppose que $\alpha$ a un élément conjugué $\beta$ de module supérieur à 2. Comme $\left|\beta \right|\geq 2$, on a, par la proposition \ref{36}, que l'ensemble des solutions irréductibles de \eqref{a} sur $<\beta>$ est $\{(0,k\beta,0,-k\beta), (k\beta,0,-k\beta,0), k \in \bZ\}$. Par la proposition \ref{41}, $\theta$ établit une bijection entre les $\lambda$-quiddités irréductibles sur $<\alpha>$ et celles sur $<\beta>$. Donc, l'ensemble des solutions irréductibles de \eqref{a} sur $<\alpha>$ est $\{(0,k\alpha,0,-k\alpha), (k\alpha,0,-k\alpha,0), k \in \bZ\}$.

\end{proof}

\noindent Le théorème \ref{26} se déduit du théorème \ref{25}.

\begin{proof}[Démonstration du théorème \ref{26}]

Soit $(a,b,k) \in \bZ^{3}$ avec $a,b \neq 0$ et $k \geq 2$. On pose $w:=a+b\sqrt{k}$ et on suppose que $w \neq \pm 1$. Si $w \in \bZ^{*}$ alors $\left|w\right| \geq 2$ et, par la proposition \ref{36}, l'ensemble des $\lambda$-quiddités irréductibles sur $<w>$ est celui donné dans l'énoncé. Si $w=0$, le résultat est déjà connu (Théorème \ref{38} i)). On suppose donc maintenant $w \notin \bZ$, c'est-à-dire $b\sqrt{k} \notin \bQ$. Quitte à remplacer $w$ par $-w$, on peut supposer que $a \geq 1$. On distingue deux cas :
\begin{itemize}
\item Si $b \geq 1$ alors $w \geq 2$. Par la proposition \ref{36}, l'ensemble des $\lambda$-quiddités irréductibles sur $<w>$ est celui donné dans l'énoncé.
\\
\item Si $b \leq -1$. Posons $P(X):=X^{2}-2aX+a^{2}-b^{2}k=(X-(a+b\sqrt{k}))(X-(a-b\sqrt{k})) \in \bZ[X]$. Comme $a+b\sqrt{k}, a-b\sqrt{k} \notin \bQ$ et $P$ est de degré 2, $P$ est irréductible sur $\bQ[X]$. Donc, $P$ est le polynôme minimal de $w$ sur $\bQ$ et $a-b\sqrt{k}$ est un élément conjugué de $w$. Or, $a-b\sqrt{k} \geq 2$. Donc, par le théorème \ref{25}, l'ensemble des $\lambda$-quiddités irréductibles sur $<w>$ est celui donné dans l'énoncé.

\end{itemize}

\end{proof}

\begin{example}

{\rm
$\{(0,k(1-\sqrt{2}),0,-k(1-\sqrt{2})), (k(1-\sqrt{2}),0,-k(1-\sqrt{2}),0), k \in \bZ\}$ est l'ensemble des $\lambda$-quiddités irréductibles sur $<1-\sqrt{2}>$.
}

\end{example}

\begin{remark}

{\rm On dispose en fait de la classification complète des $\lambda$-quiddités irréductibles sur $<a+b\sqrt{k}>$ quelles que soient les valeurs de $a, b, k$. En effet, si $b=0$ ou si $k \in \{0, 1\}$ alors la classification est donnée par le corollaire 3.11 de \cite{M6} et si $a=0$ alors la classification est donnée par le théorème \ref{38} et la proposition \ref{36}.
}
\end{remark}

\subsection{Applications}

L'objectif de cette sous-partie est d'utiliser le théorème \ref{25} pour d'obtenir d'autres résultats de classification. On commence par la proposition suivante qui n'est qu'un exemple, parmi bien d'autres, des applications directes possibles du théorème \ref{25} :

\begin{proposition}
\label{42}

Soit $w:=\frac{1-\sqrt{11}}{2}$. L'ensemble des $\lambda$-quiddités irréductibles sur $<w>$ est 
\[\{(0,kw,0,-kw), (kw,0,-kw,0), k \in \bZ\}.\]

\end{proposition}

\begin{proof}

Posons $P(X):=X^{2}-X-\frac{5}{2}=(X-\frac{1-\sqrt{11}}{2})(X-\frac{1+\sqrt{11}}{2}) \in \bQ[X]$. Comme $\frac{1-\sqrt{11}}{2}, \frac{1+\sqrt{11}}{2} \notin \bQ$ et $P$ est de degré 2, $P$ est irréductible sur $\bQ[X]$. Donc, $P$ est le polynôme minimal de $w$ sur $\bQ$ et $\frac{1+\sqrt{11}}{2}$ est un élément conjugué de $w$. Or, $\frac{1+\sqrt{11}}{2} \geq 2$. Donc, par le théorème \ref{25}, l'ensemble des $\lambda$-quiddités irréductibles sur $<w>$ est celui donné dans l'énoncé.

\end{proof}

On peut également utiliser le théorème \ref{25} pour obtenir des résultats de classification sur certains sous-groupes monogènes sans connaître précisément les éléments conjugués. Pour cela, on utilisera le résultat classique rappelé ci-dessous :

\begin{theorem}[Rouché, \cite{R} et \cite{T} corollaire 8.6.3]
\label{43}

Soient $U$ un ouvert de $\bC$, $a \in U$ et $r>0$ tels que $\overline{D(a,r)}:=\{z \in \bC,~\left|z-a\right|\leq r\} \subset U$. Soient $f$ et $g$ deux fonctions holomorphes sur $U$ vérifiant $\left|f(z)-g(z)\right|<\left|g(z)\right|$ pour tout $z$ appartenant au cercle de centre $a$ et de rayon $r$. Alors $f$ et $g$ ont le même nombre de zéros (comptés avec leur multiplicité) dans $D(a,r):=\{z \in \bC,~\left|z-a\right|< r\}$.

\end{theorem}

On aura également besoin de critères d'irréductibilité. Pour les exemples qui seront considérés, on fera usage des deux qui sont rappelés ci-dessous.

\begin{theorem}[Critère d'Eisenstein, \cite{G} Théorème I.52]
\label{44}

Soit $P(X):=\sum_{i=0}^{n} a_{i}X^{i} \in \bZ[X]$ avec $n \geq 1$ et $a_{n} \neq 0$. On suppose qu'il existe un nombre premier $p$ tel que :
\begin{itemize}
\item pour tout $i$ dans $[\![1;n-1]\!]$ $p$ divise $a_{i}$;
\item $p$ ne divise pas $a_{n}$;
\item $p^{2}$ ne divise pas $a_{0}$.
\end{itemize}
Alors $P$ est irréductible dans $\bQ[X]$.

\end{theorem}

\begin{theorem}[Critère d'irréductibilité d'Osada, \cite{Z} Théorème 1.8.6]
\label{44}

Soient $p$ un nombre premier et $P(X):=a_{n}X^{n}+\ldots+a_{1}X \pm p \in \bZ[X]$ avec $n \geq 1$ et $a_{n} \neq 0$. Si $p> \left|a_{1}\right|+\ldots+\left|a_{n}\right|$ alors $P$ est irréductible dans $\bQ[X]$.

\end{theorem}

Bien entendu, il existe une myriade d'autres résultats d'irréductibilité que l'on pourrait utiliser comme le critère de Perron (voir \cite{Z} Théorème 1.8.1), le critère de Dumas (voir \cite{D}) ou bien d'autres encore. 
\\
\\On détaille maintenant, comme annoncé, deux cas utilisant les éléments précédents :
\begin{itemize}
\item Soit $w$ une racine de $A(X):=X^{5}+10X^{4}+5X^{3}+5$. Par le critère d'Eisenstein, $A$ est irréductible sur $\bQ[X]$. Donc, $A$ est le polynôme minimal de $w$ sur $\bQ$. Posons $B(X):=10X^{4}$. Soit $z \in \bC$ avec $\left|z\right|=2$. On a :
\[\left|A(z)-B(z)\right|=\left|z^{5}+5z^{3}+5\right| \leq \left|z\right|^{5}+5\left|z\right|^{3}+5=77<160=\left|B(z)\right|.\]
\noindent Par le théorème de Rouché, $A$ a quatre racines dans $D(0,2)$. Ainsi, $A$ a une racine de module supérieur ou égal à 2. Donc, on peut classifier les $\lambda$-quiddités irréductibles sur $<w>$ avec le théorème \ref{25}.
\\
\item Soit $w$ une racine de $C(X):=X^{7}+5X^{6}+2X^{3}-X^{2}+X+11$. Par le critère d'Osada, $C$ est irréductible sur $\bQ[X]$. Donc, $C$ est le polynôme minimal de $w$ sur $\bQ$. Posons $D(X):=5X^{6}$. Soit $z \in \bC$ avec $\left|z\right|=2$. On a :
\[\left|C(z)-D(z)\right| \leq \left|z\right|^{7}+2\left|z\right|^{3}+\left|z\right|^{2}+\left|z\right|+11=161<320=\left|D(z)\right|.\]
\noindent Par le théorème de Rouché, $D$ a six racines dans $D(0,2)$. Ainsi, $D$ a une racine de module supérieur ou égal à 2. Donc, on peut classifier les $\lambda$-quiddités irréductibles sur $<w>$ avec le théorème \ref{25}.
\end{itemize} 

\begin{remarks}
{\rm 
i) Dans les deux exemples précédents, on a utilisé le théorème de Rouché. Cela dit, on peut évidemment essayer de se servir d'autres résultats de localisation des racines de polynômes, comme par exemple le théorème de Sturm (voir \cite{S}).
\\
\\ii) Les différents exemples présentés dans cette section montrent que le théorème \ref{25} permet de procéder à la classification des $\lambda$-quiddités irréductibles dans une kyrielle de cas différents qui auraient sans doute étaient difficiles à considérer avec des méthodes plus directes.
}

\end{remarks}

\section{Démonstration du théorème \ref{27}}
\label{preuve27}

On va maintenant démontrer le dernier théorème présenté dans la section \ref{RP} en utilisant une méthode proche de celle utilisée dans \cite{M6} pour démontrer le théorème de classification sur $<\sqrt{k}>$.

\begin{proof}[Preuve du théorème \ref{27}]

Soit $w:=a+ib \in \bC$ avec $\left|ab\right| \geq 1$. Notre objectif est de montrer que toutes les solutions de \eqref{a} sur $<w>$ de taille paire contiennent un zéro.
\\
\\Soient $m:=2n \geq 4$ un entier naturel pair et $(a_{1},\ldots,a_{m})$ une $\lambda$-quiddité sur $<w>$. Pour tout $i \in [\![1;m]\!]$, $a_{i}=k_{i}w$ avec $k_{i} \in \bZ$. Supposons par l'absurde qu'aucune composante de $(a_{1},\ldots,a_{m})$ soit nulle. 
\\
\\Puisque $(a_{1},\ldots,a_{m})$ est une $\lambda$-quiddité, il existe $\epsilon \in \{\pm 1 \}$ tel que \[\epsilon Id=M_{m}(a_{1},\ldots,a_{m})=\begin{pmatrix}
    K_{m}(a_{1},\ldots,a_{m}) & -K_{m-1}(a_{2},\ldots,a_{m}) \\
    K_{m-1}(a_{1},\ldots,a_{m-1})  & -K_{m-2}(a_{2},\ldots,a_{m-1}) 
   \end{pmatrix}.\]
	
\noindent Grâce à l'algorithme d'Euler, on sait que $K_{2n}(k_{1}w,k_{2}w,\ldots,k_{2n-1}w,k_{2n}w)$ est une somme de produits des $k_{j}w$ (et potentiellement de -1). Dans chaque terme d'une de ces sommes, pour chaque $j$ impair $k_{j}$ est suivi d'un $k_{l}$ avec $l$ pair. On effectue alors la manipulation suivante :
\[k_{j}wk_{l}w=k_{j}(k_{l}w^{2}).\]

\noindent Cela conduit à l'égalité suivante : \[K_{2n}(k_{1}w,k_{2}w,\ldots,k_{2n-1}w,k_{2n}w)=K_{2n}(k_{1},k_{2}w^{2},\ldots,k_{2n-1},k_{2n}w^{2}).\] 

\noindent En procédant de la même manière, on a : $K_{2n-2}(k_{2}w,\ldots,k_{2n-1}w)=K_{2n-2}(k_{2}w^{2},k_{3},\ldots,k_{2n-2}w^{2},k_{2n-1})$ (car pour pour chaque $j$ impair, $k_{j}$ est précédé d'un $k_{l}$ avec $l$ pair). 
\\
\\On procède de façon analogue pour les autres termes de $M_{m}(a_{1},\ldots,a_{m})$. Par l'algorithme d'Euler, $K_{2n-1}(k_{1}w,k_{2}w,\ldots,k_{2n-1}w)$ est une somme de produits des $k_{j}w$ (et potentiellement de -1). Pour chacun des termes de cette somme, deux possibilités existent :
\begin{itemize}
\item le terme est de la forme $k_{j}w$ avec $j$ impair;
\item le terme est un produit de $2v+1$ éléments du type $k_{j}w$ avec $v+1$ indices $j$ impair. Pour chaque $m$ pair $k_{m}$ est précédé d'un $k_{l}$ avec $l$ impair. On effectue alors la manipulation suivante :\[k_{l}wk_{m}w=k_{l}(k_{m}w^{2}).\]
\end{itemize}
\noindent On obtient ainsi $0=K_{2n-1}(k_{1}w,k_{2}w,\ldots,k_{2n-1}w)=wK_{2n-1}(k_{1},k_{2}w^{2},\ldots,k_{2n-1})$. Ceci implique que $0=K_{2n-1}(k_{1},k_{2}w^{2},\ldots,k_{2n-1})$. De la même manière, on obtient $0=K_{2n-1}(k_{2}w^{2},k_{3},\ldots,k_{2n}w^{2})$.
\\
\\Ainsi, $(k_{1},k_{2}w^{2},\ldots,k_{2n-1},k_{2n}w^{2})$ est une solution de \eqref{a} sur $\bC$.
\\
\\Par le théorème \ref{33}, cette solution contient un élément de module strictement plus petit que 2. Or, $\left|w^{2}\right|=\left|w\right|^{2}=a^{2}+b^{2} \geq 2\left|ab\right| \geq 2$ et tous les $k_{j}$ sont non nuls. Ainsi, les seules composantes de module strictement plus petit que 2 ne peuvent être que des $k_{j}$ (avec $j$ impair). Donc, il existe $j$ impair tel que $k_{j}=\pm 1$. En utilisant les deux identités $M_{3}(a,1,b)=M_{2}(a-1,b-1)$ et $M_{3}(a,-1,b)=-M_{2}(a+1,b+1)$, on peut "réduire" tous les $\pm 1$ contenus dans $(k_{1},k_{2}w^{2},\ldots,k_{2n-1},k_{2n}w^{2})$. On obtient ainsi une nouvelle solution de \eqref{a} sur $\bC$ dont les composantes sont nécessairement de la forme :
\begin{itemize}
\item $k_{j}$ avec $j$ impair et $\left|k_{j}\right| \geq 2$;
\item $k_{j}w^{2}$ avec $j$ pair;
\item $k_{j}w^{2}\pm 1$ avec $j$ pair;
\item $k_{j}w^{2}\pm 2$ avec $j$ pair.
\\
\end{itemize}

\noindent Par le théorème \ref{33}, cette nouvelle $\lambda$-quiddité contient un élément de module strictement plus petit que 2. Celui-ci appartient nécessairement aux deux dernières catégories. Calculons le module pour ces deux possibilités. On a :

\begin{itemize}
\item $\left|k_{j}w^{2}\pm 1\right|=\left|(k_{j}(a^{2}-b^{2})\pm 1) +2ik_{j}ab\right|=\sqrt{(k_{j}(a^{2}-b^{2})\pm 1)^{2}+4a^{2}b^{2}k_{j}^{2}} \geq 2\left|k_{j}\right|\left|ab\right| \geq 2$,
\item $\left|k_{j}w^{2}\pm 2\right|=\left|(k_{j}(a^{2}-b^{2})\pm 2) +2ik_{j}ab\right|=\sqrt{(k_{j}(a^{2}-b^{2})\pm 2)^{2}+4k_{j}^{2}a^{2}b^{2}} \geq 2\left|k_{j}\right|\left|ab\right| \geq 2$.
\end{itemize}

\noindent Ainsi, toutes les composantes de la solution obtenue en supprimant les $\pm 1$ dans $(k_{1},k_{2}w^{2},\ldots,k_{2n-1},k_{2n}w^{2})$ sont de module supérieur ou égal à 2. Ceci est absurde. Donc, toutes les solutions de \eqref{a} sur $<w>$ de taille paire contiennent un zéro.
\\
\\Par la proposition \ref{39}, le résultat est démontré.

\end{proof}

\begin{examples}

{\rm
On peut appliquer le théorème que nous venons de démontrer aux sous-groupes suivants : $<1+i>$, $<\frac{3}{2}-\frac{7}{8}i>$, $<\frac{1}{\sqrt{2}}+\sqrt{3} i>$.
}

\end{examples}

Dans le graphique ci-dessous (construit avec le logiciel Maxima), on a représenté le cercle de centre 0 et de rayon 2 et les courbes d'équation $y=\frac{1}{x}$ et $y=-\frac{1}{x}$. Si $w \in \bC$ est à l'extérieur du cercle alors les $\lambda$-quiddités irréductibles sur $<w>$ peuvent être classées avec la proposition \ref{36}. Si $w$ appartient aux zones hachurées alors on ne peut pas utiliser la proposition \ref{36}. En revanche, on peut classifier les $\lambda$-quiddités irréductibles sur $<w>$ avec le théorème \ref{27}.

\centerline{\includegraphics{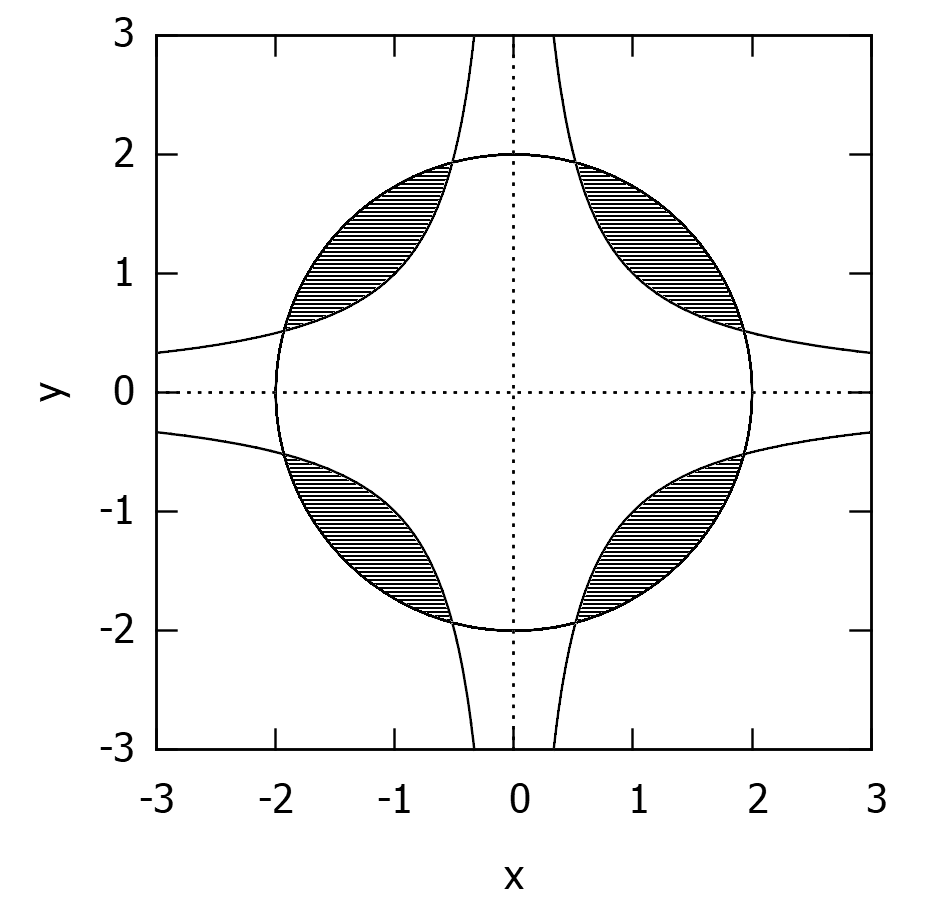}}


\begin{thebibliography}{1}

\bibitem{BR}
F. Bergeron, C. Reutenauer, 
{\it $SL_{k}$-tilings of the plane}, 
Illinois J. Math., Vol. 54 no. 1, (2010), pp 263-300.

\bibitem{Ca}
G. Cantor, 
{\it Sur une propriété du système de tous les nombres algébriques réels}, 
Acta Mathemetica, Vol. 2, (1883), pp 305-310.

\bibitem{CO}
C. Conley, V. Ovsienko, 
{\it Rotundus: triangulations, Chebyshev polynomials, and Pfaffians}, 
Math. Intelligencer, Vol. 40 no. 3, (2018), pp 45-50.

\bibitem{Cox}
H. S. M. Coxeter, 
{\it Frieze patterns},
Acta Arithmetica, Vol. 18, (1971), pp 297-310.

\bibitem{C} 
M. Cuntz,
{\it A combinatorial model for tame frieze patterns}, 
Munster J. Math., Vol. 12 no. 1, (2019), pp 49-56.

\bibitem{CH} 
M. Cuntz, T. Holm,
{\it Frieze patterns over integers and other subsets of the complex numbers}, 
J. Comb. Algebra., Vol. 3 no. 2, (2019), pp 153-188.

\bibitem{D}
G. Dumas,
{\it Sur quelques cas d’irréductibilité des polynômes à coefficients rationnels},
Journal de mathématiques pures et appliquées 6$^{e}$ série, tome 2, (1906), pp 191-258.

\bibitem{G} 
I. Gozard,
{\it Théorie de Galois - niveau L3-M1 - 2e édition}, 
Ellipses, 2009.

\bibitem{M1}
F. Mabilat,
\textit{Combinatoire des sous-groupes de congruence du groupe modulaire}, 
Annales Mathématiques Blaise Pascal, Vol. 28 no. 1, (2021), pp. 7-43. doi : 10.5802/ambp.398. https://ambp.centre-mersenne.org/articles/10.5802/ambp.398/.

\bibitem{M2}
F. Mabilat,
\textit{$\lambda$-quiddité sur $\mathbb{Z}[\alpha]$ avec $\alpha$ transcendant},
Mathematica Scandinavica, Vol. 128 no. 1, (2022), pp 5-13, https://doi.org/10.7146/math.scand.a-128972.

\bibitem{M3}
F. Mabilat,
{\it Combinatoire des sous-groupes de congruence du groupe modulaire II}. 
Annales Mathématiques Blaise Pascal, Vol. 28 no. 2, (2021), pp. 199-229. doi : 10.5802/ambp.404. https://ambp.centre-mersenne.org/articles/10.5802/ambp.404/.

\bibitem{M4}
F. Mabilat,
\textit{Entiers monomialement irréductibles}, hal-03487145, arXiv:2112.10410.

\bibitem{M5}
F. Mabilat,
\textit{Solutions monomiales minimales irréductibles dans $SL_{2}(\mathbb{Z}/p^{n}\mathbb{Z})$},
Bulletin des Sciences Mathématiques, Vol. 194, (2024), Article 103456, ISSN 0007-4497, https://doi.org/10.1016/j.bulsci.2024.103456.

\bibitem{M6}
F. Mabilat,
\textit{$\lambda$-quiddité sur certains sous-groupes monogènes de $\bC$},
Confluentes mathematici, Vol. 17, (2025), pp. 11-31.

\bibitem{Mo1}
S. Morier-Genoud,
\textit{Coxeter's frieze patterns at the crossroad of algebra, geometry and combinatorics}, 
Bull. Lond. Math. Soc., Vol. 47 no. 6, (2015), pp 895-938.

\bibitem{O} 
V. Ovsienko, 
{\it Partitions of unity in $SL(2,\mathbb{Z})$,  negative continued fractions,  and dissections of polygons,} 
Res. Math. Sci., Vol. 5 no. 2, (2018), Article 21, 25 pp. 

\bibitem{R}
E. Rouché,
{\it Mémoire sur la série de Lagrange},
J. de l’École Impériale polytechnique 39, tome XXII, (1862), pp 193–224.

\bibitem{S}
C. Sturm,
{\it Mémoire sur la résolution des équations numériques},
Mémoires présentés par divers Savants étrangers à l'Acad. royale des sc., section sc. math, phys., tome VI, (1835), pp 273-318.

\bibitem{T}
P. Tauvel,
{\it Analyse complexe pour la Licence 3},
Dunod, 2020.

\bibitem{WZ}
M. Weber, M. Zhao,
{\it Factorization of frieze patterns,}
Revista de la Unión Matemática Argentina, Vol. 60 no. 2, (2019), pp 407-415.

\bibitem{Z}
K. Zhao,
{\it Ring And Field Theory},
World Scientific Publishing Co, 2022.


\end{thebibliography}
\end{document}